\documentclass{amsart}
\usepackage{color}

\newcommand{\bM}{\bar{M}}

\newcommand{\ol}{\overline}
\newcommand{\ul}{\underline}

\newtheorem{theorem}{Theorem}[section]
\newtheorem{lemma}[theorem]{Lemma}

 \theoremstyle{definition}

\theoremstyle{remark}
\newtheorem{remark}[theorem]{Remark}

\numberwithin{equation}{section}

\begin{document}

\title[Hessian equations of parabolic type]
{The first initial boundary value problem
for Hessian equations of parabolic type
on Riemannian manifolds}

\author{Weisong Dong}
\address{Department of Mathematics, Harbin Institute of Technology,
         Harbin, 150001, China}
\email{dweeson@gmail.com}
\author{Heming Jiao}
\address{Department of Mathematics, Harbin Institute of Technology,
         Harbin, 150001, China}
\email{jiao@hit.edu.cn}


\begin{abstract}
In this paper, we are concerned with the first initial boundary
value problem for a class of fully nonlinear parabolic equations
on Riemannian manifolds. As usual, the establishment of the
\emph{a priori} $C^2$ estimates is our main part. Based on
these estimates, the existence of classical solutions is proved
under conditions which are nearly optimal.



\end{abstract}

\maketitle

\section{Introduction}

In this paper, we study the Hessian equations of parabolic
type of the form
\begin{equation}
\label{eqn}
f (\lambda(\nabla^{2}u + \chi), - u_t)  = \psi (x, t)
\end{equation}
in $M_T = M \times (0,T] \subset M \times \mathbb{R}$
satisfying the boundary condition
\begin{equation}
\label{eqn-b}
 u = \varphi \; \mbox{on $\mathcal{P}M_T$},
\end{equation}
where $(M, g)$ is a compact Riemannian manifold of
dimension $n \geq 2$ with smooth boundary $\partial M$
and $\bM := M \cup \partial M$, $\mathcal{P} M_T = B M_T \cup S M_T$
is the parabolic boundary of $M_T$ with $B M_T = M \times \{0\}$ and
$S M_T = \partial M \times [0,T]$, $f$ is a symmetric smooth function
of $n + 1$ variables defined in an open convex symmetric cone
$\Gamma \subset \mathbb{R}^{n + 1}$ with vertex at the origin and
\[\Gamma_{n + 1} \equiv \{\lambda \in \mathbb{R}^{n + 1}:\mbox{ each component }
   \lambda_i > 0, \; 1 \leq i \leq n+1\} \subseteq \Gamma,\]
$\nabla^2 u$ denotes the Hessian of $u (x, t)$ with respect to $x \in M$,
$u_t = \frac{\partial u}{\partial t}$ is the derivative of $u (x, t)$ with respect
to $t \in [0, T]$, $\chi$ is a smooth
(0, 2) tensor on $\bM$ and $\lambda (\nabla^{2} u + \chi) = (\widehat{\lambda}_1 ,\ldots,\widehat{\lambda}_n)$
denotes the eigenvalues of $\nabla^{2} u + \chi$ with respect to the metric $g$.

As in \cite{CNS} (see \cite{G} also), we assume that $f$ satisfies the following structural conditions:
\begin{equation}
\label{f1}
f_i \equiv \frac{\partial f}{\partial \lambda_i} > 0 \mbox{ in } \Gamma,\ \ 1\leq i \leq n + 1,
\end{equation}
\begin{equation}
\label{f2}
f\mbox{ is concave in }\Gamma
\end{equation}
and
\begin{equation}
\label{f5}
\delta_{\psi, f} \equiv \inf_M \psi - \sup_{\partial \Gamma} f > 0, \mbox{ where }
   \sup_{\partial \Gamma} f \equiv \sup_{\lambda_0 \in \partial \Gamma}
      \limsup_{\lambda \rightarrow \lambda_0} f (\lambda).
\end{equation}

We mean an admissible function by $u \in C^{2}(M_T)$ satisfying
$(\lambda(\nabla^{2} u + \chi), - u_t) \in \Gamma$ in $M_T$,
where $C^{k}(M_T)$ denotes the space of functions defined on
$M_T$ which are $k$-times continuously differentiable with
respect to $x \in M$ and $[k/2]$-times continuously differentiable
with respect to $t \in (0, T]$ and $[k/2]$ is the largest integer not
greater than $k/2$. We note that \eqref{eqn} is parabolic for
admissible solutions (see \cite{CNS}).

We first recall the following notations
\[
|u|_{C^k (\ol{M_T})} = \sum_{|\beta| + 2r \leq k} \sup_{\ol{M_T}} |\nabla^\beta D^r_t u|,
\]
\[
\begin{aligned}
|u|_{C^{k + \alpha} (\ol{M_T})} & = |u|_{C^k (\ol{M_T})} \\
   + & \sup_{|\beta| + 2r = k}
   \sup_{\begin{array}{c}
   (x, s), (y, t) \in \ol{M_T} \\ (x, s) \neq (y, t) \end{array}}
     \frac{|\nabla^\beta D^r_t u (x, s) - \nabla^\beta D^r_t u (y, t)|}{\Big(|x - y| + |s - t|^{1/2}\Big)^\alpha}
\end{aligned}
\]
and $C^{k + \alpha} (\ol{M_T})$ denotes the subspace of $C^{k} (\ol{M_T})$ defined by
\[
C^{k + \alpha} (\ol{M_T}) : = \{u \in C^{k} (\ol{M_T}): |u|_{C^{k + \alpha} (\ol{M_T})} < \infty\}.
\]

In the current paper, we are interested in the existence of admissible solutions
to \eqref{eqn}-\eqref{eqn-b}. The key step is to establish the \emph{a priori} $C^2$ estimates.
Using the methods from \cite{G2}, where Guan studied the elliptic counterpart of \eqref{eqn}:
\begin{equation}
\label{eqn-elliptic}
f (\lambda(\nabla^{2}u + \chi))  = \psi (x)
\end{equation}
in $M$ satisfying the Dirichlet boundary condition, we are able to obtain these
estimates under nearly minimal restrictions on $f$.

Our main results are stated in the following theorem.
\begin{theorem}
\label{dj-th3}
Suppose that $\psi \in C^{\infty} (\ol{M_T})$, $\varphi \in C^{\infty} (\ol{\mathcal{P} M_T})$ for
$0 < T \leq \infty$,
\begin{equation}
\label{comp1}
(\lambda (\nabla^2 \varphi (x, 0) + \chi(x)), - \varphi_t (x, 0)) \in \Gamma
   \mbox{ for all } x \in \ol M
\end{equation}
and
\begin{equation}
\label{comp}
f (\lambda (\nabla^2 \varphi (x, 0) + \chi(x)), - \varphi_t (x, 0))  = \psi (x, 0)
   \mbox{ for all } x \in \partial M.
\end{equation}
In addition to \eqref{f1}-\eqref{f5}, assume that
\begin{equation}
\label{f6}
f_j (\lambda) \geq \nu_0 \Big(1 + \sum_{i = 1}^{n+1} f_i (\lambda)\Big) \mbox{ for any }
  \lambda \in \Gamma \mbox{ with } \lambda_j < 0,
\end{equation}
for some positive constant $\nu_0$,
\begin{equation}
\label{c-290}
\sum_{i =1}^{n+1} f_i \lambda_i \geq - K_0 \sum_{i = 1}^{n + 1} f_i, \; \; \forall \lambda \in \Gamma
\end{equation}
for some $K_0 \geq 0$ and that there exists an admissible subsolution $\ul u \in C^2 (\ol{M_T})$
satisfying
\begin{equation}
\label{sub}
\left\{ \begin{aligned}
f(\lambda(\nabla^{2} \underline{u} + \chi), - \underline{u}_{t})  \geq \,& \psi(x,t) & \mbox{ in } M_T,\\
\ul u = \,& \varphi  & \mbox{ on } SM_T,\\
\ul u \leq \,& \varphi  & \mbox { on } BM_T.
\end{aligned} \right.
\end{equation}
Then there exists a unique admissible solution
$u \in C^{\infty} (\ol{M_T})$ of \eqref{eqn}-\eqref{eqn-b}.
\end{theorem}

\begin{remark}
Condition \eqref{f6} is only used to derive the gradient estimates as many authors,
see \cite{L}, \cite{GS91}, \cite{G0}, \cite{Li91}, \cite{Tr90} and \cite{U} for examples.

Condition \eqref{c-290} is used in the estimates for both $|\nabla u|$ and $|u_t|$.
We will see that in the gradient estimates, condition \eqref{c-290} can be weakened by
\begin{equation}
\label{c-290'}
\sum_{i =1}^{n+1} f_i \lambda_i \geq
   - K_0 \Big(1 + \sum_{i = 1}^{n + 1} f_i\Big), \; \; \forall \lambda \in \Gamma.
\end{equation}

As in \cite{G2}, the existence of $\ul u$ is useful to construct some barrier functions
which are crucial to our estimates.
\end{remark}

The most typical examples of $f$ satisfying the conditions in Theorem \ref{dj-th3}
are $f = \sigma^{1/k}_k$ and $f = (\sigma_k / \sigma_l)^{1/(k - l)}$,
$1 \leq l < k \leq n + 1$, defined in the G{\aa}rding cone
\[
\Gamma_{k} = \{\lambda \in \mathbb{R}^{n + 1}: \sigma_{j} (\lambda) > 0, j = 1, \ldots, k\},
\]
where $\sigma_{k}$ are the elementary symmetric functions
\[
\sigma_{k} (\lambda) = \sum_ {i_{1} < \ldots < i_{k}}
\lambda_{i_{1}} \ldots \lambda_{i_{k}},\ \ k = 1, \ldots, n + 1.
\]

When $f = \sigma^{1/(n + 1)}_{n + 1}$, equation \eqref{eqn} can be written as
the parabolic Monge-Amp\`{e}re equation:
\begin{equation}
\label{MA-p}
- u_t \det(\nabla^2 u + \chi) = \psi^{n + 1},
\end{equation}
which was introduced by Krylov in \cite{K1} when $\chi = 0$ in
Euclidean space. Our motivation to study \eqref{eqn} is from their
natural connection to the deformation of surfaces by
some curvature functions. For example,
equation \eqref{MA-p} plays a key role in the study of
contraction of surfaces by Gauss-Kronecker curvature
(see Firey \cite{Firey} and Tso \cite{Tso}).
For the study of more general curvature flows, the
reader is referred to \cite{A}, \cite{AMZ}, \cite{Han97},
\cite{M} and their references.
\eqref{MA-p} is also relevant to a
maximum principle for parabolic equations (see Tso \cite{Tso1}).

In \cite{L}, Lieberman studied the first initial-boundary value
problem of equation \eqref{eqn} when $\chi \equiv 0$ and
$\psi$ may depend on $u$ and $\nabla u$ in a bounded domain
$\Omega \subset \mathbb{R}^{n + 1}$ under various conditions.
Jiao and Sui \cite{JS} considered the parabolic Hessian equation
of the form
\begin{equation}
\label{eqn2}
f (\lambda (\nabla^2 u + \chi)) - u_t = \psi (x, t)
\end{equation}
on Riemannian manifolds using techniques from \cite{G2} and
\cite{GJ} where the authors studied the corresponding elliptic
equations. Guan, Shi and Sui \cite{GSS} extended the work
of \cite{JS} using the idea of \cite{G2}; they also treated
the parabolic equation of the form
\begin{equation}
\label{eqn3}
f (\lambda (\nabla^2 u + \chi)) = e^{u_t + \psi}.
\end{equation}
Applying the methods of \cite{G},
Bao and Dong \cite{BD} solved \eqref{eqn}-\eqref{eqn-b} under an
additional condition which is introduced in \cite{G} (see \cite{GJ} also)
\begin{equation}
\label{gj-condition}
T_{\lambda} \cap \partial \Gamma^{\sigma} \; \mbox{is a
nonempty compact set, $\forall \lambda \in \Gamma$ and $\sup_{\partial \Gamma} f < \sigma < f (\lambda)$},
\end{equation}
where $\partial \Gamma^{\sigma} = \{\lambda \in \Gamma: f(\lambda) = \sigma\}$ is
the boundary of $\Gamma^{\sigma} = \{\lambda \in \Gamma: f (\lambda) > \sigma\}$
and $T_{\lambda}$ denote
the tangent plane at $\lambda$ of $\partial \Gamma^{f (\lambda)}$,
for $\sigma > \sup_{\partial \Gamma} f$ and $\lambda \in \Gamma$.
The reader is referred to \cite{Li90}, \cite{U}, \cite{G0}, \cite{G}, \cite{G2},
\cite{GJ}, \cite{GJ2} and their references for the study of elliptic Hessian
equations on manifolds.

We can prove the short time existence as Theorem 15.9 in \cite{L}.
So without of loss of generality, we may assume
\begin{equation}
\label{comp-0}
f (\lambda (\nabla^2 \varphi (x, 0) + \chi(x)), - \varphi_t (x, 0))  = \psi (x, 0)
   \mbox{ for all } x \in \ol M.
\end{equation}

As usual, the main part of this paper is to derive the \emph{a priori}
$C^2$ estimates. We see that \eqref{eqn} is uniformly parabolic after
establishing the $C^2$ estimates by \eqref{f1} and \eqref{f5}.
The $C^{2, \alpha}$ estimates can be obtained by applying
Evans-Krylov theorem (see \cite{Evans} and \cite{K}).
Finally Theorem \ref{dj-th3} can be proved as
Theorem 15.9 of \cite{L}.

The rest of this paper is organized as follows.
In section 2, we introduce some notations and
useful lemmas. $C^1$ estimates are derived in
Section 3. An \emph{a priori} bound for $|u_t|$
is obtained in Section 4. Section 5 and Section 6
are devoted to the global and boundary estimates
for second order derivatives respectively.

\section{Preliminaries}
\label{gj-P}
\setcounter{equation}{0}
\medskip

Let $F$ be the function defined by $F (A, \tau) = f (\lambda (A), \tau)$
for $A \in \mathbb{S}^{n}$, $\tau \in \mathbb{R}$ with $(\lambda (A), \tau) \in \Gamma$,
where $\mathbb{S}^{n}$ is the set of $n \times n$ symmetric matrices.
It was shown in \cite{CNS} that $F$ is concave from \eqref{f2}.
For simplicity we shall use the notations $U = \nabla^2 u + \chi$,
$\ul U = \nabla^2 \ul u + \chi$
and under an orthonormal local frame $e_1, \ldots, e_{n}$,
\[ U_{ij} \equiv U (e_i, e_j) = \nabla_{ij} u + \chi_{ij}, \;\;
 \ul U_{ij} \equiv \ul U (e_i, e_j) = \nabla_{ij} \ul u + \chi_{ij}.
\]
Thus, \eqref{eqn} can be written in the form locally
\begin{equation}
\label{eqn0}
F (U, -u_t) = f (\lambda (U_{ij}), - u_t) = \psi.
\end{equation}
Let
\[
\begin{aligned}
F^{ij} = \frac{\partial F}{\partial A_{ij}} (U, - u_t), \,&
   F^{\tau} = \frac{\partial F}{\partial \tau} (U, - u_t)\\
     F^{ij, kl} = \frac{\partial^2 F}{\partial A_{ij} \partial A_{kl}} (U, - u_t), \,&
        F^{ij, \tau} = \frac{\partial^2 F}{\partial A_{ij} \partial \tau} (U, - u_t).
\end{aligned}
\]
By \eqref{f1} we see that $F^\tau > 0$ and $\{F^{ij}\}$ is positive definite.
We shall also denote the eigenvalues of $\{F^{ij}\}$ by
$f_1, \ldots, f_n$ when there
is no possible confusion. We note that $\{U_{ij}\}$ and $\{F^{ij}\}$
can be diagnolized simultaneously and that
\[F^{ij} U_{ij} = \sum f_i \widehat{\lambda}_i, \;\; F^{ij} U_{ik} U_{kj} = \sum f_i \widehat{\lambda}_i^2,\]
where $\lambda (\{U_{ij}\}) = (\widehat{\lambda}_1, \ldots, \widehat{\lambda}_{n})$.

Similarly to \cite{G2}, we write
\[\begin{aligned}\mu (x, t) = & (\lambda (\ul U (x, t)), - \ul u_t (x, t)), \\
  \lambda(x, t) = & (\lambda(U (x, t)), - u_t (x, t))
\end{aligned}\]
and $\nu_\lambda \equiv Df(\lambda)/|Df(\lambda)|$ is the unit normal vector to the level
hypersurface $\partial \Gamma^{f (\lambda)}$ for $\lambda \in \Gamma$.
Since $K \equiv \{\mu (x, t):  (x, t) \in \overline{M_T}\}$ is a compact subset of $\Gamma$,
there exist uniform constants $\beta \in (0, \frac{1}{2 \sqrt{n + 1}})$ such that
\begin{equation}
\label{G}
\nu_{\mu (x, t)} - 2 \beta \mathbf{1} \in \Gamma_{n + 1}, \forall (x, t) \in \overline {M_T}.
\end{equation}
where $\mathbf{1} = (1, \cdots, 1) \in \mathbb{R}^{n + 1}$ (see \cite{G2}).

We need the following Lemma which is proved in \cite{G2}.
\begin{lemma}
\label{G-Lemma}
Suppose that $|\nu_\mu - \nu_\lambda| \geq \beta$.
Then there exists a uniform constant $\varepsilon > 0$ such that
\begin{equation}
\label{3I-100}
\sum_{i=1}^{n + 1} f_i (\lambda) (\mu_{i} - \lambda_i) \geq
     \varepsilon \Big(1 + \sum_{i=1}^{n + 1} f_i (\lambda) \Big).
\end{equation}
\end{lemma}

Define the linear operator $\mathcal{L}$ locally by
\[
\mathcal{L} v = F^{ij} \nabla_{ij} v - F^\tau v_t, \mbox{ for } v \in C^2 (\ol{M_T}).
\]

From Lemma \ref{G-Lemma} and Lemma 6.2 of \cite{CNS} it is easy to derive that
when $|\nu_{\mu (x, t)} - \nu_{\lambda (x, t)}| \geq \beta$,
\begin{equation}
\label{gj}
\mathcal{L} (\ul u - u) \geq \varepsilon \Big(1+ \sum F^{ii} + F^{\tau} \Big).
\end{equation}

If $|\nu_\mu - \nu_\lambda| < \beta$, we have
$\nu_{\lambda} - \beta \mathbf{1} \in \Gamma_{n + 1}$. It follows that
\begin{equation}
\label{Guan}
f_{i} \geq \frac{\beta}{\sqrt{n + 1}} \sum_{j = 1}^{n + 1} f_{j}, \;
\forall 1 \leq i \leq n + 1.
\end{equation}

\section{The $C^1$ estimates}

Since $u$ is admissible and
$\Gamma \subset \{ \lambda \in \mathbb{R}^{n + 1}: \sum_{i =1}^{n + 1} \lambda_i > 0\}$,
we see that $u$ is a subsolution of
\begin{equation}
\label{ups}
\left\{\begin{aligned} \triangle h - h_t + \mathrm{tr} (\chi) = & 0, \; & \mbox{ in } \; M_T,\\
  h = & \varphi, \; & \mbox{ on } \; \mathcal{P} M_T.
\end{aligned}\right.
\end{equation}
Let $h$ be the solution of \eqref{ups}. It follows from the maximum principle that
$\ul u \leq u \leq h$ on $\overline {M_T}$.
Therefore, we have
\begin{equation}
\label{C0}
\sup_{\overline{M_T}} |u| + \sup_{\mathcal{P} M_T} |\nabla u| \leq C.
\end{equation}
For the global gradient estimates, we can prove the following theorem.
\begin{theorem}
\label{gradient}
Suppose that \eqref{f1}, \eqref{f2}, \eqref{f6} and \eqref{c-290'} hold.
Let $u \in C^{3}(\overline {M_T})$ be an
admissible solution of (\ref{eqn}) in $M_T$.
Then
\begin{equation}
\label{gsui-2}
\sup_{\overline{M_T}}|\nabla u| \leq C ( 1 + \sup_{\mathcal{P} M_T}|\nabla u|),
\end{equation}
where $C$ depends on $|\psi|_{C^1 (\overline {M_T})}$,
$|u|_{C^0 (\overline {M_T})}$ and other known data.
\end{theorem}
\begin{proof}
Set
\[W = \sup_{(x, t) \in \overline{M_T}} w e^\phi,\]
where $w = \frac{|\nabla u|^2}{2}$ and $\phi$ is a function to be determined.
It suffices to estimate $W$ and we may assume that $W$ is achieved at
$(x_{0}, t_{0}) \in \overline {M_T} - \mathcal{P} M_T$.
Choose a smooth orthonormal local frame $e_1, \ldots, e_n$
about $x_0$ such that $\nabla_{e_i} e_j = 0$ at $x_0$
and $U (x_0, t_0)$ is diagonal. We see that the function $\log w + \phi$
attains its maximum at $(x_0, t_0)$.
Therefore, at $(x_0, t_0)$, we have
\begin{equation}
\label{g1}
\frac{\nabla_i w}{w} + \nabla_i \phi = 0, \; \mbox{ for each }i = 1, \ldots, n,
\end{equation}
\begin{equation}
\label{g11}
\frac{w_t}{w} + \phi_t \geq 0
\end{equation}
and
\begin{equation}
\label{g2}
\frac{\nabla_{ii} w}{w} - \Big(\frac{\nabla_i w}{w}\Big)^2 + \nabla_{ii} \phi \leq 0.
\end{equation}
Differentiating the equation \eqref{eqn}, we get
\begin{equation}
\label{g4}
F^{ii} \nabla_{k} U_{ii} - F^\tau \nabla_k u_t = \nabla_k \psi \mbox{ for }
   k = 1, \ldots, n
\end{equation}
and
\begin{equation}
\label{g44}
F^{ii} (U_{ii})_t - F^\tau u_{tt} = \psi_t.
\end{equation}
Note that
\begin{equation}
\label{g3}
\nabla_i w = \nabla_k u \nabla_{ik} u, \  w_t = \nabla_k u (\nabla_k u)_t,\
   \nabla_{ii} w = (\nabla_{ik} u)^2 + \nabla_k u \nabla_{iik} u
\end{equation}
and that
\begin{equation}
\label{hess-A70}
 \nabla_{ijk} u - \nabla_{jik} u = R^l_{kij} \nabla_l u.
\end{equation}
We have, by \eqref{g11}, \eqref{g4}, \eqref{g3} and \eqref{hess-A70},
\begin{equation}
\label{g5}
\begin{aligned}
F^{ii} \nabla_{ii} w \geq \,& \nabla_k u F^{ii} \nabla_{iik} u\\
    \geq \,& - C |\nabla u| - C |\nabla u|^2 \sum F^{ii} + F^\tau \nabla_k u \nabla_k u_t\\
       \geq \,& - C |\nabla u| - C |\nabla u|^2 \sum F^{ii} - w F^\tau \phi_t,
\end{aligned}
\end{equation}
provided $|\nabla u|$ is sufficiently large.
Combining \eqref{g1}, \eqref{g2}, \eqref{g5}, we obtain
\begin{equation}
\label{g6}
0 \geq - \frac{C}{|\nabla u|} - C \sum F^{ii} - F^{ii} (\nabla_i \phi)^2 + \mathcal{L} \phi.
\end{equation}
Let $\phi = \delta v^2$, where $v = u + \sup_{\overline{M_T}} |u| + 1$ and
$\delta$ is a small positive constant to be chosen. Thus, choosing
$\delta$ sufficiently small such that $2 \delta - 4 \delta^2 v^2 \geq c_0 > 0$
for some uniform constant $c_0$, by \eqref{c-290'},
\begin{equation}
\label{g7}
\begin{aligned}
\mathcal{L} \phi - F^{ii} (\nabla_i \phi)^2 = \,& 2 \delta v (F^{ii} \nabla_{ii} u - F^\tau u_t)
   + (2 \delta - 4 \delta^2 v^2) F^{ii} (\nabla_i u)^2\\
     \geq \,& - C \delta \Big(1 + \sum F^{ii} + F^\tau\Big) + c_0 F^{ii} (\nabla_i u)^2.
\end{aligned}
\end{equation}
It follows from \eqref{g6} and \eqref{g7} that
\begin{equation}
\label{g8}
c_0 F^{ii} (\nabla_i u)^2 \leq C \Big(1 + \sum F^{ii} + F^\tau\Big),
\end{equation}
provided $|\nabla u|$ is sufficiently large. We may assume
$|\nabla u (x_0, t_0)| \leq n \nabla_1 u (x_0, t_0)$ and by
\eqref{g1},
\[
U_{11} = - 2 \delta v w + \frac{\nabla_k u \chi_{1k}}{\nabla_1 u} < 0
\]
provided $w$ is sufficiently large. Then we can derive from \eqref{f6} that
\[
F^{11} \geq \nu_0 \Big(1 + \sum F^{ii} + F^\tau\Big).
\]
Therefore, we obtain a bound $|\nabla u (x_0, t_0)| \leq C n^2/c_0 \nu_0$
by \eqref{g8} and \eqref{gsui-2} is proved.
\end{proof}
\begin{remark}
We see that in the proof of Theorem \ref{gradient}, we do not need
the existence of $\ul u$.
\end{remark}
By \eqref{C0} and \eqref{gsui-2}, the $C^1$ estimates are established.

\section{estimate for $|u_t|$ }

In this section, we derive the estimate for $|u_t|$.
\begin{theorem}
\label{ut}
Suppose that \eqref{f1}, \eqref{f2}, \eqref{c-290} and \eqref{sub} hold.
Let $u \in C^{3}(\overline {M_T})$ be an
admissible solution of (\ref{eqn}) in $M_T$. Then there
exists a positive constant $C$ depending on $|u|_{C^1 (\ol{M_T})}$,
$|\ul u|_{C^2 (\ol{M_T})}$, $|\psi|_{C^2 (\overline {M_T})}$ and
other known data such that
\begin{equation}
\label{utesti}
\sup_{\ol{M_T}} |u_t| \leq C (1 + \sup_{\mathcal{P} M_T} |u_t|).
\end{equation}
\end{theorem}
\begin{proof}
We first show that
\begin{equation}
\label{utesti1}
\sup_{\ol{M_T}} (- u_t) \leq C (1 + \sup_{\mathcal{P} M_T} (- u_t))
\end{equation}
for which we set
\[
W = \sup_{\overline{M_T}} (- u_t) e^\phi,
\]
where $\phi$ is a function to be chosen. We may assume that $W$ is
attained at $(x_0, t_0) \in \overline{M_T} - \mathcal{P} M_T$. As in
the proof of Theorem \ref{gradient}, we choose an orthonormal local
frame $e_1, \cdots, e_n$ about $x_0$ such that $\nabla_{e_i} e_j = 0$ and
$\{U_{ij} (x_0, t_0)\}$ is diagonal. We may assume $- u_t (x_0, t_0) > 0$.
At $(x_0, t_0)$ where the function $\log (- u_t) + \phi$ achieves its maximum,
we have
\begin{equation}
\label{dj-1}
\frac{\nabla_i u_t}{u_t} + \nabla_i \phi = 0, \mbox{ for each } i = 1, \ldots, n,
\end{equation}
\begin{equation}
\label{dj-2}
\frac{u_{tt}}{u_t} + \phi_t \geq 0,
\end{equation}
and
\begin{equation}
\label{dj-3}
0 \geq F^{ii} \Big\{\frac{\nabla_{ii} u_t}{u_t} - \Big(\frac{\nabla_i u_t}{u_t}\Big)^2
  + \nabla_{ii} \phi\Big\}.
\end{equation}
Combining \eqref{dj-1}, \eqref{dj-2} and \eqref{dj-3}, we find
\begin{equation}
\label{dj-4}
0 \geq \frac{1}{u_t} (F^{ii} \nabla_{ii} u_t - F^\tau u_{tt}) - F^{ii} (\nabla_i \phi)^2
    + \mathcal{L} \phi.
\end{equation}
By \eqref{g44} and \eqref{dj-4},
\begin{equation}
\label{dj-6}
\mathcal{L} \phi \leq - \frac{\psi_t}{u_t} + F^{ii} (\nabla_i \phi)^2.
\end{equation}
Let $\phi =  \frac{\delta^2}{2} |\nabla u|^2 + \delta u + b (\ul u - u)$,
where $\delta \ll b \ll 1$ are positive constants to be determined. By straightforward calculations, we see
\[\begin{aligned}
\nabla_i \phi = & \delta^2 \nabla_k u \nabla_{ik} u + \delta \nabla_i u + b \nabla_i (\ul u - u),\\
\phi_t = & \delta^2 \nabla_k u (\nabla_{k} u)_t + \delta u_t + b (\ul u - u)_t, \\
\nabla_{ii} \phi = & \delta^2 (\nabla_{ik} u)^2 + \delta^2 \nabla_k u \nabla_{iik} u + \delta \nabla_{ii} u + b \nabla_{ii} (\ul u - u).
\end{aligned}\]
It follows that,
in view of \eqref{g4} and \eqref{hess-A70},
\begin{equation}
\label{dj-7}
\begin{aligned}
\mathcal{L} \phi \geq \,& \delta^2 \nabla_k u (F^{ii} \nabla_{iik} u
  - F^\tau (\nabla_k u)_t) + \frac{\delta^2}{2} F^{ii} U_{ii}^2 \\
   & + \delta F^{ii} \nabla_{ii} u - \delta F^\tau u_t - C \delta^2 \sum F^{ii} + b \mathcal{L} (\ul u - u)\\
     \geq \,& - C \delta^2 \Big(1 + \sum F^{ii}\Big)  + \frac{\delta^2}{2} F^{ii} U_{ii}^2
      + \delta \mathcal{L} u + b \mathcal{L} (\ul u - u).
\end{aligned}
\end{equation}
Next,
\begin{equation}
\label{dj-10}
(\nabla_i \phi)^2 \leq C \delta^4 U_{ii}^2 + C b^2.
\end{equation}
Thus, we can derive from \eqref{dj-6}, \eqref{dj-7} and \eqref{dj-10} that
\begin{equation}
\label{dj-8}
\begin{aligned}
b \mathcal{L} (\ul u - u) + \frac{\delta^2}{4} F^{ii} U_{ii}^2
   + \delta \mathcal{L} u
   \leq  - \frac{C}{u_t} + C \delta^2 \Big(1 + \sum F^{ii}\Big) + C b^2 \sum F^{ii},
\end{aligned}
\end{equation}
when $\delta$ is small enough.
Now we use the idea of \cite{G2} to consider two cases: (i) $|\nu_{\mu_0} - \nu_{\lambda_0}| \geq \beta$
and (ii) $|\nu_{\mu_0} - \nu_{\lambda_0}| < \beta$, where $\mu_0 = \mu (x_0, t_0)$
and $\lambda_0 = \lambda (x_0, t_0)$.

In case (i), by Lemma \ref{G-Lemma}, we see that \eqref{gj} holds.
By \eqref{c-290}, we have
\begin{equation}
\label{con}
\mathcal{L} u \geq F^{ii} U_{ii} - F^\tau u_t - C \sum F^{ii} \geq
    - K_0 \Big(\sum F^{ii} + F^{\tau}\Big) - C \sum F^{ii}.
\end{equation}
Combining with \eqref{con} and \eqref{dj-8}, we have
\begin{equation}
\label{dj-11}
\begin{aligned}
 b \mathcal{L} (\ul u - u)
   \leq  - \frac{C}{u_t} + C \delta \Big(1 + \sum F^{ii} + F^\tau\Big)
     + C b^2 \sum F^{ii}.
\end{aligned}
\end{equation}
Now using \eqref{gj} we can choose $\delta \ll b \ll 1$ to
obtain a bound $- u_t (x_0, t_0) \leq \frac{C}{b \varepsilon}$.

In case (ii), we see that \eqref{Guan} holds.
By \eqref{dj-8}, we have
\begin{equation}
\label{dj-13}
\begin{aligned}
b \mathcal{L} (\ul u - u) + & \frac{\delta^2}{4} F^{ii} U_{ii}^2
   + \delta (F^{ii} U_{ii} - F^\tau u_t)\\
   \leq & - \frac{C}{u_t} + C \delta^2 \Big(1 + \sum F^{ii}\Big) + C (\delta + b^2 ) \sum F^{ii}.
\end{aligned}
\end{equation}
Note that
\begin{equation}
\label{dj-14}
\frac{\delta^2}{4} F^{ii} U_{ii}^2
\geq \delta F^{ii} |U_{ii}| -  \sum F^{ii}
\end{equation}
and
\begin{equation}
\label{dj-5}
\mathcal{L} (\ul u - u) \geq 0
\end{equation}
by the concavity of $F$.
Therefore, by \eqref{dj-13}, \eqref{dj-14} and \eqref{dj-5}, we have
\begin{equation}
\label{dj-12}
\begin{aligned}
 - \delta F^\tau u_t
   \leq - \frac{C}{u_t} + C \delta^2 + C \sum F^{ii}.
\end{aligned}
\end{equation}
By \eqref{c-290}, similar to \cite{G2},
\begin{equation}
\label{dj-9}
\begin{aligned}
- u_t \Big(\sum F^{ii} + F^\tau\Big)
     \geq \,& f (- u_t {\bf 1}) - f (\lambda (U), - u_t)
        + \sum F^{ii} U_{ii} - F^\tau u_t\\
     \geq \,& f (- u_t {\bf 1}) - f (\lambda (\ul U), - \ul u_t)
       - K_0 \Big(\sum F^{ii} + F^\tau\Big)\\
     \geq \,& 2 b_0 + u_t \Big(\sum F^{ii} + F^\tau\Big)
\end{aligned}
\end{equation}
for some uniform constant $b_0 > 0$, provided $- u_t$ is sufficiently large, where
${\bf 1} = (1, \ldots, 1) \in \mathbb{R}^{n + 1}$. It follows that, by \eqref{Guan},
\[
- F^\tau u_t \geq \frac{- \beta u_t}{\sqrt{n+1}} \Big(\sum F^{ii} + F^\tau\Big)
\geq \frac{- \beta u_t}{2 \sqrt{n+1}} \Big(\sum F^{ii} + F^\tau\Big)
  + \frac{\beta b_0}{2 \sqrt{n+1}}.
\]
Choose $\delta$ sufficiently small such that
\[
\frac{\beta b_0 \delta}{2 \sqrt{n+1}} - C \delta^2 \geq c_1 > 0
\]
for some constant $c_1$.
Therefore, we can derive from \eqref{dj-12} that
\[
- u_t (x_0, t_0) \leq \max\Big\{\frac{C \sqrt{n+1}}{\beta \delta},
 \frac{C}{c_1}\Big\}.
\]
So \eqref{utesti1} holds.

Similarly, we can show
\begin{equation}
\label{utesti2}
\sup_{\overline{M_T}} u_t \leq C (1 + \sup_{\mathcal{P} M_T} u_t)
\end{equation}
by setting
\[
W = \sup_{\bM_T} u_t e^\phi
\]
and $\phi = \frac{\delta^2}{2} |\nabla u|^2 - \delta u + b (\ul u - u)$.

Combining \eqref{utesti1} and \eqref{utesti2}, we can see that \eqref{utesti} holds.
\end{proof}

Since $u_t = \varphi_t$ on $S M_T$ and \eqref{comp-0}, we can derive the estimate
\begin{equation}
\label{ut0}
\sup_{\ol{M_T}} |u_t| \leq C.
\end{equation}

\section{Global estimates for second order derivatives}

In this section, we derive the global estimates for the second order derivatives.
We prove the following maximum principle.
\begin{theorem}
\label{dj-th1}
Let $u \in C^{4} (\overline{M_T})$ be an
admissible solution of (\ref{eqn}) in $M_T$.
Suppose that \eqref{f1}, \eqref{f2} and \eqref{sub} hold.
Then
\begin{equation}
\label{gsui-1}
\sup_{\overline {M_T}} |\nabla^2 u| \leq C (1 + \sup_{\mathcal{P} M_T} |\nabla^2 u|),
\end{equation}
where $C > 0$ depends on $|u|_{C^1 (\overline {M_T})}$,
$|u_t|_{C^0(\overline {M_T}))}$, $|\psi|_{C^{2}(\overline {M_T}))}$
and other known data.
\end{theorem}
\begin{proof}
Set
\[W = \max_{(x,t) \in \overline{M_T}} \max_{\xi \in T_x M, |\xi| = 1}
     (\nabla_{\xi\xi} u + \chi(\xi, \xi)) e^\phi,\]
where $\phi$ is a function to be determined.
We may assume $W$ is achieved at $(x_{0}, t_{0}) \in \ol{M_T} - \mathcal{P} M_T$
and $\xi_0 \in T_{x_0} M$.
Choose a smooth orthonormal local frame $e_{1}, \ldots, e_{n}$ about $x_{0}$
as before
such that $\xi_0 = e_1$, $\nabla_{e_i} e_j = 0$,  and
$\{U_{ij} (x_0, t_0)\}$ is diagonal. We see that
$W = U_{11} (x_0, t_0) e^{\phi (x_0, t_0)}$.
We may also assume that $U_{11} \geq \ldots \geq U_{nn}$  at $(x_0, t_0)$.

Since the function $\log (U_{11}) + \phi$ attains its maximum at $(x_0, t_0)$,
we have, at $(x_0, t_0)$,
\begin{equation}
\label{gs3}
\frac{\nabla_i U_{11}}{U_{11}} + \nabla_i \phi = 0
    \mbox{ for each } i = 1, \ldots, n,
\end{equation}
\begin{equation}
\label{gs4}
\frac{(\nabla_{11} u)_t}{U_{11}} + \phi_t \geq 0,
\end{equation}
and
\begin{equation}
\label{gs5}
0 \geq \sum_i F^{ii} \Big\{\frac{\nabla_{ii} U_{11}}{U_{11}}
   - \Big(\frac{\nabla_i U_{11}}{U_{11}}\Big)^2 + \nabla_{ii} \phi\Big\}.
\end{equation}
Therefore, by \eqref{gs4} and \eqref{gs5}, we find
\begin{equation}
\label{gs6}
\mathcal{L} \phi \leq - \frac{1}{U_{11}} (F^{ii} \nabla_{ii} U_{11} - F^\tau (\nabla_{11} u)_t)
   + F^{ii} \Big(\frac{\nabla_i U_{11}}{U_{11}}\Big)^2.
\end{equation}
By the formula
\begin{equation}
\label{hess-A80}
\begin{aligned}
\nabla_{ijkl} v - \nabla_{klij} v
= & R^m_{ljk} \nabla_{im} v  + \nabla_i R^m_{ljk} \nabla_m v
      + R^m_{lik} \nabla_{jm} v \\
  & + R^m_{jik} \nabla_{lm} v
      + R^m_{jil} \nabla_{km} v + \nabla_k R^m_{jil} \nabla_m v
\end{aligned}
\end{equation}
we have
\begin{equation}
\label{chu2}
  \begin{aligned}
   \nabla_{ii}U_{11} \geq \nabla_{11} U_{ii} - C U_{11},
\end{aligned}
\end{equation}
Differentiating equation (\ref{eqn}) twice, we have
\begin{equation}
\begin{aligned}
\label{gs2}
& F^{ij} \nabla_{11} U_{ij} - F^\tau \nabla_{11} u_t
+ F^{ij,kl} \nabla_1 U_{ij} \nabla_1 U_{kl}\\
+ & F^{\tau\tau} (\nabla_1 u_t)^2 - 2 F^{ij, \tau} \nabla_1 U_{ij} \nabla_1 u_t
  = \nabla_{11} \psi \geq - C.
\end{aligned}
\end{equation}
It follows from \eqref{gs6}, \eqref{chu2} and \eqref{gs2} that
\begin{equation}
\label{gs11}
 \mathcal{L} \phi \leq \frac{C}{U_{11}} + C \sum F^{ii} + E,
\end{equation}
where
\[
E = \frac{1}{U_{11}} \Big(F^{ij,kl} \nabla_1 U_{ij} \nabla_1 U_{kl}
 - 2 F^{ij, \tau} \nabla_1 U_{ij} \nabla_1 u_t + F^{\tau\tau} (\nabla_1 u_t)^2\Big)
   + F^{ii} \Big(\frac{\nabla_{i} U_{11}}{U_{11}}\Big)^2.
\]
$E$ can be estimated as in \cite{G} using an idea of Urbas \cite{U} to which
the following inequality proved by Andrews \cite{A} and Gerhardt \cite{GC} is crucial.
\begin{lemma}
\label{AG-Lemma}
For any symmetric matrix $\eta=\{\eta_{ij}\}$ we have
\[
F^{ij,kl}\eta_{ij}\eta_{kl}=\sum_{i,j}\frac{\partial^{2}f}{\partial\lambda_{i}\partial\lambda_{j}}\eta_{ii}\eta_{jj}
   +\sum_{i\neq j}\frac{f_{i}-f_{j}}{\lambda_{i}-\lambda_{j}}\eta_{ij}^{2}.
\]
The second term on the right hand side is nonpositive if $f$ is concave, and is interpreted as a limit if $\lambda_{i}=\lambda_{j}$.
\end{lemma}
Similar to \cite{G}, we can derive (see \cite{BD} also)
\begin{equation}
\label{gs8}
 E \leq \sum_{i \in J} F^{ii} (\nabla_i \phi)^2
   + C \sum_{i \in K} F^{ii} + C F^{11} \sum_{i \in K} (\nabla_i \phi)^2,
\end{equation}
where $J = \{i: 3 U_{ii} \leq - U_{11}\}$ and
$K = \{i:  3 U_{ii} > - U_{11}\}$.

Let
\[
\phi = \frac{\delta |\nabla u|^2}{2} + b (\ul u - u),
\]
where $\delta$ and $b$ are positive constants to be determined.
Thus, we can derive from \eqref{gs8} that
\begin{equation}
\label{gs88}
 E \leq C b^2 \sum_{i \in J} F^{ii}  + C \delta^2 F^{ii} U_{ii}^2
            +  C \sum_{i \in K} F^{ii} + C (\delta^2 U_{11}^2 + b^2) F^{11}.
\end{equation}

On the other hand, by \eqref{g4} and \eqref{hess-A70},
\begin{equation}
\label{gs9}
\begin{aligned}
\mathcal{L} \phi = \,& \delta F^{ii} \sum_k (\nabla_{ik} u)^2 + \delta \nabla_k u F^{ii} \nabla_{iik} u
  - \delta \nabla_k u F^\tau (\nabla_k u)_t + b \mathcal{L} (\ul u - u)\\
     \geq \,& \delta F^{ii} U_{ii}^2 + b \mathcal{L} (\ul u - u) - C \delta \Big(1 + \sum F^{ii}\Big)
\end{aligned}
\end{equation}
Combining \eqref{gs11}, \eqref{gs88} and \eqref{gs9}, we obtain
\begin{equation}
\label{gs10}
\frac{\delta}{2} F^{ii} U_{ii}^2 + b \mathcal{L} (\ul u - u)
   \leq \frac{C}{U_{11}} + C b^2 \sum_{i \in J} F^{ii} + C b^2 F^{11} + C \Big(1 + \sum F^{ii}\Big)
\end{equation}
provided $\delta$ is sufficiently small. Note that $|U_{jj}| \geq \frac{1}{3} U_{11}$, for
$j \in J$. Therefore, by \eqref{gs10}, we have
\begin{equation}
\label{gs12}
\frac{\delta}{4} F^{ii} U_{ii}^2 + b \mathcal{L} (\ul u - u)
   \leq C \Big(1 + \sum F^{ii}\Big)
\end{equation}
when $U^2_{11} \geq \max\{C b^2/\delta, 1\}$.

Now let $\mu_0 = \mu (x_0, t_0)$ and $\lambda_0 = \lambda (x_0, t_0)$.
If $|\lambda_0 - \mu_0| \geq \beta$, we can obtain a bound of
$U_{11} (x_0, t_0)$ by \eqref{gj} as in \cite{G}.

If $|\lambda_0 - \mu_0| < \beta$, we see that \eqref{Guan} holds.
Let $\widehat{\lambda} = \lambda (U (x_0, t_0))$. We may assume
$|\widehat{\lambda}| \geq |u_t (x_0, t_0)|$. Similar
to \cite{G2}, by the concavity of $f$,
\begin{equation}
\label{gs13}
\begin{aligned}
|\widehat{\lambda}| \Big(\sum F^{ii} + F^\tau\Big)
     \geq \,& f (|\widehat{\lambda}| {\bf 1}) - f (\lambda (U), - u_t)
        + \sum F^{ii} U_{ii} - F^\tau u_t\\
     \geq \,& f (|\widehat{\lambda}| {\bf 1}) - f (\lambda (\ul U), - \ul u_t)
       - |\widehat{\lambda}| \Big(\sum F^{ii} + F^\tau\Big)\\
     \geq \,& 2 b_0 - |\widehat{\lambda}| \Big(\sum F^{ii} + F^\tau\Big)
\end{aligned}
\end{equation}
for some uniform positive constant $b_0$, provided $|\widehat{\lambda}|$
is sufficiently large. By \eqref{Guan}, \eqref{dj-5} and \eqref{gs12},
we see that
\begin{equation}
\label{gs14}
2 c_0 |\widehat{\lambda}|^2 \Big(\sum F^{ii} + F^\tau\Big)
   \leq C \Big(1 + \sum F^{ii}\Big),
\end{equation}
where
\[
c_0 := \frac{\delta \beta}{8\sqrt {n + 1}}.
\]
Then we can derive a bound of $|\widehat{\lambda}|$ from \eqref{gs13}.
\end{proof}

\section{Boundary estimates for second order derivatives}

In this section, we consider the estimates of second order derivatives on
$S M_T$. We may assume
$\varphi \in C^{4} (\overline{M_T})$.
For simplicity we shall make use of the condition \eqref{c-290'} though stronger
results may be proved (see \cite{G}, \cite{G2} and \cite{GJ2}).

The pure tangential second derivatives are easy to estimate from the
boundary condition $u = \varphi$ on $\mathcal{P} M_T$. So we are
focused on the estimates for mixed tangential-normal and pure normal second
derivatives.

Fix a point $(x_{0}, t_{0}) \in S M_T$. We shall choose
smooth orthonormal local frames $e_1, \ldots, e_n$ around $x_0$ such that
when restricted to $\partial M$, $e_n$ is normal to $\partial M$.

Let $\rho (x)$ and $d (x)$ denote the distance from $x \in M$ to $x_{0}$
and $\partial M$ respectively and set
\[M_T^{\delta} = \{X = (x, t) \in M \times (0,T]:
      \rho (x) < \delta \}.\]
We shall use the following barrier function as in \cite{G}.
\begin{equation}
\label{hess-E176}
 \varPsi
   = A_1 v + A_2 \rho^2 - A_3 \sum_{\gamma < n} |\nabla_{\gamma} (u - \varphi)|^2,
\end{equation}
where
\[v = (u - \underline{u}) + ad - \frac{Nd^{2}}{2}.\]

Now we show the following lemma which is useful to construct barrier functions (see Lemma
\ref{lem1}).
\begin{lemma}
\label{lem2}
Suppose \eqref{f2} and \eqref{c-290'} hold. Then for any $\sigma > 0$ and
any index $r$,
\begin{equation}
\label{djs-1}
\sum f_i |\widehat{\lambda}_i| \leq \sigma \sum_{i \neq r} f_i \widehat{\lambda}^2_i
   + \frac{C}{\sigma} \Big(\sum f_i + F^\tau\Big) + C
\end{equation}
\end{lemma}
\begin{proof}
If $\widehat{\lambda}_r < 0$, by \eqref{c-290'}, we see
\[
\begin{aligned}
\sum f_i |\widehat{\lambda}_i| = \,& 2 \sum_{\widehat{\lambda}_i > 0} f_i \widehat{\lambda}_i
   - \sum f_i \widehat{\lambda}_i + F^\tau u_t - F^\tau u_t\\
     \leq \,& \sigma \sum_{\widehat{\lambda}_i > 0} f_i \widehat{\lambda}^2_i + \frac{1}{\sigma} \sum_{\widehat{\lambda}_i > 0} f_i
        + K_0 \Big(1 + \sum f_i + F^\tau\Big) + C F^\tau
\end{aligned}
\]
and \eqref{djs-1} follows.

If $\widehat{\lambda}_r \geq 0$, by the concavity of $f$,
\[
\begin{aligned}
\sum f_i |\widehat{\lambda}_i| = \,& \sum f_i \widehat{\lambda}_i
    - 2 \sum_{\widehat{\lambda}_i < 0} f_i \widehat{\lambda}_i \\
     \leq \,& \sigma \sum_{\widehat{\lambda}_i < 0} f_i \widehat{\lambda}^2_i + \frac{1}{\sigma} \sum_{\widehat{\lambda}_i < 0} f_i
        + \sum f_i \widehat{\mu}_i - F^\tau (\ul u - u)_t\\
          \leq \,& \sigma \sum_{\widehat{\lambda}_i < 0} f_i \widehat{\lambda}^2_i + \frac{1}{\sigma} \sum_{\widehat{\lambda}_i < 0} f_i
              + C \Big(\sum f_i + F^\tau\Big).
\end{aligned}
\]
Then \eqref{djs-1} is proved.
\end{proof}

The following Lemma is crucial to our estimates and the idea is
mainly from \cite{G2} and \cite{GJ} (see \cite{GSS} also).
\begin{lemma}
\label{lem1}
Suppose that \eqref{f1}, \eqref{f2} and \eqref{sub} hold. Then for any constant
$K > 0$, there exist uniform positive constants $a, \delta$ sufficiently small,
and $A_1, A_2, A_3, N$ sufficiently large such that $\varPsi \geq K (d + \rho^2)$
in $\overline{M_T^\delta}$ and
\begin{equation}
\label{sb-L}
\mathcal{L} \varPsi \leq - K \Big(1 + \sum_{i = 1}^n f_i |\widehat{\lambda}_i|
    + \sum_{i = 1}^n f_i+ F^\tau \Big) \mbox{ in }M_T^{\delta}.
\end{equation}
\end{lemma}
\begin{proof}
For any fixed $(x, t) \in M^{\delta}_T$, we may assume that $U_{ij}$ and $F^{ij}$ are both
diagonal at $(x, t)$.
Firstly, we have (see \cite{G} for details),
\begin{equation}
\label{bs1}
\begin{aligned}
\mathcal{L} (\nabla_k (u - \varphi))  \leq \,&
   C \Big(1 + \sum f_i |\widehat{\lambda}_i| + \sum f_i + F^\tau\Big),
 \;\; \forall \; 1 \leq k \leq n.
\end{aligned}
\end{equation}
Therefore,
\begin{equation}
\label{bs2}
\begin{aligned}
 \sum_{l < n} \mathcal{L}  (|\nabla_l (u - \varphi)|^2)
 \geq \, \sum_{l < n} F^{ij} U_{i l} U_{j l}
         - C \Big(1 + \sum f_i |\widehat{\lambda}_i| + \sum f_i + F^\tau\Big).
\end{aligned}
\end{equation}
Using the same proof of Proposition 2.19 in \cite{G}, we can show
\begin{equation}
\label{eq3-7}
\sum_{l < n} F^{ij} U_{i l} U_{j l}
     \geq \frac{1}{2} \sum_{i \neq r} f_i \widehat{\lambda}_i^2,
\end{equation}
for some index $r$. Write $\mu = \mu (x, t)$ and $\lambda = \lambda (x, t)$ and
note that $\mu = (\widehat{\mu}, - \ul u_t)$ and $\lambda = (\widehat{\lambda}, - u_t)$, where
$\widehat{\mu} = \lambda (\ul U)$.

We shall consider two cases as before: $\mathbf{(a)}$ $|\nu_{\mu} - \nu_{\lambda}| < \beta$ and
$\mathbf{(b)}$ $|\nu_{\mu} - \nu_{\lambda}| \geq \beta$.

Case {\bf (a)}.  By \eqref{Guan}, we have
\begin{equation}
\label{bs3}
 f_i \geq \frac{\beta}{\sqrt{n + 1}} \Big(\sum f_k + F^\tau\Big),
    \;\; \forall \, 1 \leq i \leq n.
\end{equation}

Now we make a little modification of the proof of Lemma 3.1 in \cite{G2}
to show the following inequality
\begin{equation}
\label{eq3-8}
\sum_{i \neq r} f_i \widehat{\lambda}_i^2 
     \geq c_0 \sum f_i \widehat{\lambda}_i^2 - C_0 \Big(\sum f_i + F^\tau\Big)
\end{equation}
for some $c_0, C_0 > 0$. If $\widehat{\lambda}_r < 0$, we have
\begin{equation}
\label{bs4}
\widehat{\lambda}_r^2 \leq n \sum_{i \neq r} \widehat{\lambda}_i^2 + C,
\end{equation}
where $C$ depends on the bound of $u_t$ since
\[
\sum \widehat{\lambda}_i - u_t > 0.
\]
Therefore, by \eqref{bs3} and \eqref{bs4}, we have
\begin{equation}
\label{bs5}
f_r \widehat{\lambda}_r^2 \leq n f_r \sum_{i \neq r} \widehat{\lambda}_i^2
  + C f_r \leq \frac{n \sqrt{n + 1}}{\beta} \sum_{i \neq r} f_i \widehat{\lambda}_i^2 + C \sum f_i
\end{equation}
and \eqref{eq3-8} holds.

Now suppose $\widehat{\lambda}_r \geq 0$. By the concavity of $f$,
\begin{equation}
\label{bs6}
f_r \widehat{\lambda}_r \leq f_r \widehat{\mu}_r - F^\tau (\ul u_t - u_t)
   + \sum_{i \neq r} f_i (\widehat{\mu}_i - \widehat{\lambda}_i).
\end{equation}
Thus, by \eqref{bs3} and Schwarz inequality, we have
\begin{equation}
\label{bs7}
\begin{aligned}
& \frac{\beta f_r \widehat{\lambda}^2_r}{\sqrt{n + 1}} \Big(\sum f_i + F^\tau\Big)\\
   \leq \,& f^2_r \widehat{\lambda}^2_r
     \leq C \Big(f^2_r \widehat{\mu}^2_r + \sum_{k \neq r} f_k \sum_{i \neq r}
        f_i (\widehat{\mu}_i^2 + \widehat{\lambda}_i^2) + (F^\tau)^2\Big)\\
          \leq \,& C \Big(\sum f_i + F^\tau\Big) \Big\{\Big(\sum f_i + F^\tau\Big)
            + \sum_{i \neq r} f_i \widehat{\lambda}_i^2\Big\},
\end{aligned}
\end{equation}
where $C$ may depend on the bound of $|u_t|$. It follows that
\begin{equation}
\label{bs8}
f_r \widehat{\lambda}^2_r \leq C \sum_{i \neq r} f_i \widehat{\lambda}_i^2
   + C \Big(\sum f_i + F^\tau\Big)
\end{equation}
and \eqref{eq3-8} holds.

We first suppose $|\lambda| \geq R$ for $R$ sufficiently large.
By \eqref{gs13}, we see
\begin{equation}
\label{bs11}
\sum f_i \widehat{\lambda}_i^2 \geq b_0 |\widehat{\lambda}|
\end{equation}
when $R$ is sufficiently large.
Since $|\nabla d| \equiv 1$, when $a$ and $\delta$ are sufficiently
small, by \eqref{bs3}, we have,
\begin{equation}
\label{bs15}
\begin{aligned}
\mathcal{L} v \leq \,& \Big(\mathcal{L} (u - \ul u) + C_0 (a + N d) \sum f_i
  - N F^{ij} \nabla_i d \nabla_j d\Big)\\
     \leq \,& - \frac{\beta N}{2\sqrt{n + 1}} \Big(\sum f_k + F^\tau\Big).
\end{aligned}
\end{equation}
Note that for any $\sigma > 0$,
\begin{equation}
\label{bs16}
\sum f_i |\widehat{\lambda}_i| \leq \sigma \sum f_i \widehat{\lambda}_i^2 + \frac{1}{\sigma} \sum f_i.
\end{equation}
Therefore, it follows from \eqref{eq3-8}, \eqref{bs15} and \eqref{bs16} that
for any $\sigma > 0$,
\begin{equation}
\label{bs10}
\begin{aligned}
\mathcal{L} \varPsi \leq \,&- \frac{\beta A_1 N}{2\sqrt{n + 1}} \Big(\sum f_k + F^\tau\Big) + C A_2 \sum f_i \\
& - \frac{A_3}{2} \sum_{i \neq r} f_i \widehat{\lambda}_i^2
    + C A_3 \Big(1 + \sum f_i |\widehat{\lambda}_i| + \sum f_i + F^\tau\Big)\\
\leq \,& - \frac{\beta A_1 N}{2\sqrt{n + 1}} \Big(\sum f_k + F^\tau\Big)
     - \frac{A_3 c_0}{2} \sum f_i \widehat{\lambda}_i^2 + C A_2 \sum f_i\\
& + C A_3 \Big(1 + \sum f_i |\widehat{\lambda}_i| + \sum f_i + F^\tau\Big)\\
\leq \,& - \frac{\beta A_1 N}{2\sqrt{n + 1}} \Big(\sum f_k + F^\tau\Big)
    + \Big(A_3 \sigma - \frac{A_3 c_0}{2}\Big) \sum f_i \widehat{\lambda}_i^2 \\
& + C \Big(A_2 + \frac{A_3}{\sigma}\Big) \sum f_i + C A_3 \Big(1 + F^\tau\Big).
\end{aligned}
\end{equation}
Let $\sigma = c_0/4$, we find
\begin{equation}
\label{bs12}
\begin{aligned}
\mathcal{L} \varPsi
\leq \,& - \frac{\beta A_1 N}{2\sqrt{n + 1}} \Big(\sum f_k + F^\tau\Big)
    - \frac{A_3 c_0}{4} \sum f_i \widehat{\lambda}_i^2 \\
& + C (A_2 + A_3)\Big(\sum f_i + F^\tau \Big) + C A_3\\
\leq \,& - \frac{\beta A_1 N}{2\sqrt{n + 1}} \Big(\sum f_k + F^\tau\Big)
    - \frac{A_3 c_0 b_0}{8} |\widehat{\lambda}| -  A_3 \sum f_i |\widehat{\lambda}_i| \\
& + C (A_2 + A_3)\Big(\sum f_i + F^\tau \Big) + C A_3\\
\leq \,& - \frac{\beta A_1 N}{2\sqrt{n + 1}} \Big(\sum f_k + F^\tau\Big)
    -  A_3 \Big(1 +\sum f_i |\widehat{\lambda}_i|\Big)\\
& + C (A_2 + A_3)\Big(\sum f_i + F^\tau \Big)
\end{aligned}
\end{equation}
by choosing $R \geq 8C/c_0 b_0 + 1$.

If $|\widehat{\lambda}| \leq R$, by \eqref{f1} and \eqref{f5}, we have
\[
c_1 I \leq \{F^{ij}\} \leq C_1, \ \ c_1 \leq F^\tau \leq C_1
\]
for some uniform positive constants $c_1$, $C_1$ which may depend
on $R$. Therefore, we have
\begin{equation}
 \label{bs13}
 \mathcal{L} \varPsi \leq C (- A_1 +  A_2 +  A_3) \Big(1 + \sum f_i
     + \sum f_i |\widehat{\lambda_i}| + F^\tau\Big)
\end{equation}
where $C$ depends on $c_1$ and $C_1$.

Case ({\bf b}). By Lemma \ref{G-Lemma}, we may fix $a$ and $\delta$ sufficiently
small such that $v \geq 0$ in $M^\delta_T$ and
\begin{equation}
\label{v2}
\mathcal{L} v \leq - \frac{\varepsilon}{2} \Big(1 + \sum f_i + F^\tau \Big) \;\; \mbox{in} \;\; M_T^\delta.
\end{equation}
Thus, by Lemma \ref{lem2}, we have
\begin{equation}
\label{bs14}
\begin{aligned}
\mathcal{L} \varPsi
 \leq \,& - \frac{\varepsilon A_1}{2} \Big(1 + \sum f_i + F^\tau \Big) + C A_2 \sum f_i
    - \frac{A_3}{2} \sum_{i \neq r} f_i \widehat{\lambda}_i^2 \\
      & + C A_3 \Big(1 + \sum f_i + \sum f_i |\widehat{\lambda}_i|\Big)\\
 \leq \,& \Big(- \frac{\varepsilon A_1}{2} + C A_2 + C A_3\Big) \Big(1 + \sum f_i + F^\tau \Big)
      - \frac{A_3}{4} \sum_{i \neq r} f_i \widehat{\lambda}_i^2\\
 \leq \,& \Big(- \frac{\varepsilon A_1}{2} + C A_2 + C A_3\Big) \Big(1 + \sum f_i + F^\tau \Big)
      - A_3 \sum f_i |\widehat{\lambda}_i|.
\end{aligned}
\end{equation}
Checking \eqref{bs12}, \eqref{bs13} and \eqref{bs14}, we can choose $A_1 \gg A_2 \gg A_3 \gg 1$
such that \eqref{sb-L} holds and  $\varPsi \geq K (d + \rho^2)$ in $\overline{M_T^\delta}$.
Therefore, Lemma \ref{lem1} is proved.
\end{proof}

The estimates for mixed tangential-normal second derivatives
can be established immediately using $\varPsi$ as a barrier
function by \eqref{bs1} and the maximum principle (see \cite{BD}).

The pure normal second derivatives can be derived as \cite{G}
using an idea of Trudinger \cite{T}. The reader is referred
to \cite{BD} for details.

\end{document}